\documentclass[12pt]{amsart}
\usepackage[all]{xy}
\usepackage{amssymb}

\newcommand{\bbf}{\mathbb{F}}
\newcommand{\bbn}{\mathbb{N}}

\newcommand{\bbq}{\mathbb{Q}}

\newtheorem{thm}{Theorem}

\newtheorem{lem}[thm]{Lemma}
\newtheorem{prop}[thm]{Proposition}

\newtheorem{defn}[thm]{Definition}

\newtheorem{remarkk}[thm]{Remark}

\newtheorem{examplee}[thm]{Example}

\title{Parallelopipeds of Positive Rank Twists of Elliptic Curves}
\author{Bo-Hae Im and Michael Larsen}
\address{Department of Mathematics, Chung-Ang University, 221, Heukseok-dong, Dongjak-gu, Seoul, 156-756, South Korea}\email{bohaeim@gmail.com}
\address{Department of Mathematics, Indiana University, Bloomington,
Indiana 47405, USA} \email{larsen@math.indiana.edu}
\thanks{Michael Larsen was partially supported by NSF grant DMS-0800705.}

\begin{document}
\begin{abstract} For every $n\in \bbn$ there exists an elliptic curve $E/\bbq$ and an $n$-dimensional subspace $V$ of $\bbq^\times/(\bbq^\times)^2$ such that for all $v\in V$,
the quadratic twist $E_v$ has positive rank.
\end{abstract}
\maketitle

\section{Introduction}
Let $E$ be an elliptic curve over a number field $K$.
For every element $v\in K^\times/(K^\times)^2$, there is a well-defined quadratic twist
of $E$ associated to $v$, which we denote $E_v$.
The main result of this paper is the following:
\begin{thm}\label{main}
For every $n\in \bbn$ there exists an elliptic curve $E/\bbq$ and an $n$-dimensional subspace $V$ of $\bbq^\times/(\bbq^\times)^2$ such that for all $v\in V$,
the quadratic twist $E_v$ has positive rank.
\end{thm}
We can express this theorem slightly differently, along the lines of \cite{IL}. 
If $L/K$ is a finite Galois extension of number fields, with group $G$ and $E/K$ is an elliptic curve, then $G$ acts on $V := E(L)\otimes \bbq$.  If $W$ is any $G$-subrepresentation of $V$, we say that the pair $(G,W)$ is \emph{Mordell-Weil over $K$}.
Theorem~\ref{main} now asserts that the regular representation of every elementary $2$-group is Mordell-Weil over $\bbq$.  
We remark that T.~Dokchitser and V.~Dokchitser \cite{DD}
have given examples of number fields $K$ such that,
modulo the Birch-Swinnerton-Dyer conjecture,
every regular representation of an elementary $2$-group is Mordell-Weil over $K$ in dimension $1$.  In fact, something stronger is true: their examples have the remarkable property that (under Birch-Swinnerton-Dyer) \emph{every} quadratic twist has positive rank.  (Of course, $K$ cannot be $\bbq$.)

Our proof depends on Vatsal's theorem \cite{Va} which asserts that for the elliptic curve $X_0(19)$
the set of positive integers $d$ such that $E_d$ has rank $1$ has positive density.

\section{Some Multiplicative Combinatorics}

\begin{defn}
Let $n$ be a positive integer.
We say that a subset $\Pi$ of $\bbn$
is a \emph{strict $n$-parallelopiped} if we have
$$\Pi = \{c\prod_{i\in I} a_i\mid I\subset \{1,2,\ldots,n\}\}$$
for some rational numbers $a_i,c\in \bbq^*$ such that the images of $a_1,\ldots,a_n$
are linearly independent in the $\bbf_2$-vector space $\bbq^*/{\bbq^*}^2$.
\end{defn}

The main result of this section is the following:
\begin{thm}
\label{paras}
Any set $S$ of positive integers of positive lower density contains a strict $n$-parallelopiped for each positive integer $n$.
\end{thm}

For any subset $S$ of $\bbn$, we define the function
$$f_S(t) := \sum_{n\in S} e^{-nt}$$
on the domain $(0,\infty)$.  Thus, $S\subset T$ implies that for all $t$, $f_S(t) \le f_T(t).$
In particular,
\begin{equation}
\label{upper}
f_S(t) \le f_{\bbn}(t) = \frac 1{1-e^{-t}} \le 2\max(1/t,1).
\end{equation}
More generally, this principle applies to multisets.
If $S_1,\ldots,S_n\subset \bbn$, then for all $t\in (0,\infty)$, the first two Bonferroni inequalities
imply
\begin{equation}
\label{bon-1}
f_{S_1\cup\cdots\cup S_n}(t) \le \sum_{1\le i\le n} f_{S_i}(t)
\end{equation}
and
\begin{equation}
\label{bon-2}
\sum_{1\le i\le n} f_{S_i}(t) - \sum_{1\le i < j\le n} f_{S_i\cap S_j}(t) \le f_{S_1\cup\cdots\cup S_n}(t).
\end{equation}

If $a\le b$, we define
$$P_{a,b} = \prod_{a\le p\le b} p,$$
where the product is taken over primes.  For any positive integer $m$, we define
$$R(m) := \{n\in \bbn\mid (m,n)=1\}$$
the set of positive integers relatively prime to $m$.

\begin{lem}
\label{sieve}
For any $S\subset \bbn$ and $b\ge a>1$ and for all $t\in (0,\infty)$, we have
$$f_{S\cap R(P_{a,b})} = \sum_{\{n\in S\mid (n,P_{a,b})=1\}} e^{-nt}
\le t^{-1}\prod_{a\le p\le b} (1-p^{-1}).$$
\end{lem}

\begin{proof}
Clearly
$$f_{S\cap R(P_{a,b})}(t) \le f_{R_{a,b}}(t) = \sum_{i=0}^\infty \sum_{\{n\in (iP_{a,b},(i+1)P_{a,b} )
\mid (n,P_{a,b})=1\}} e^{-nt}.$$
As
$$\sum_{\{n\in (iP_{a,b},(i+1)P_{a,b} ) \mid (n,P_{a,b})=1\}} e^{-nt}
\le \phi(P_{a,b}) e^{-iP_{a,b}t},$$
we have
$$f_{R_{a,b}}(t) \le \frac{\phi(P_{a,b})}{1-e^{-P_{a,b}t}} \le  \frac {\phi(P_{a,b})}{tP_{a,b}}
= t^{-1}\prod_{a\le p\le b} (1-p^{-1}).$$
\end{proof}

For any $S\subset \bbn$ we define
$$D(S) := \limsup_{T\to\infty} \frac 1{\log T}\int_{1/T}^1 f_S(t)\,dt.$$
\begin{lem}
\label{compare-d}
For all subsets $S$ of $\bbn$, let $d(S)$ denote the lower density of $S$. Then
$$D(S) \ge d(S).$$
\end{lem}

\begin{proof}
For all $\epsilon > 0$, there exists $N$ such that for all $n\ge N$,
$$|S\cap [1,n]| > (d(S)-\epsilon)n.$$
Therefore, for all $t>0$,
$$f_S(t) > \sum_{n=N}^\infty (d(S)-\epsilon) e^{-nt} = \frac{e^{-Nt}(d(S)-\epsilon)}{1-e^{-t}}\ge
\frac{e^{-Nt}(d(S)-\epsilon)}t.$$
Thus, for $T > N/\epsilon$, we have
\begin{align*}
\frac 1{\log T}\int_{1/T}^1 f_S(t)\,dt
&\ge \frac 1{\log T} \int_{1/T}^{\epsilon/N} \frac{e^{-\epsilon}(d(S)-\epsilon)\,dt}{t} \\
&\ge \Bigl(1-\frac {\log N/\epsilon}{\log T}\Bigr)e^{-\epsilon}(d(S)-\epsilon).
\end{align*}
Taking $\limsup_{T\to \infty}$, we conclude that for all $\epsilon>0$,
$$D(S) \ge e^{-\epsilon}(d(S)-\epsilon),$$
which implies the lemma.
\end{proof}

For any subset $S$ of $\bbn$ and any positive integer $m$, let
$$S_m := \bbn\cap m^{-1} S = \{n\in \bbn\mid mn\in S.\}$$

\begin{prop}
\label{ind-step}
If $S$ is a subset of $\bbn$ with $D(S) > 0$, and $N$ is a positive integer, there exist
primes $q>p>N$ such that
$$D(S_p\cap S_q) > 0.$$
\end{prop}

\begin{proof}

Let $a$ be an integer larger than $\max(N,\frac{12}{D(S)})$.  As
$$\prod_{p\ge a\;\text{prime}}(1-p^{-1}) = \prod_{p < a\;\text{prime}} (1-p^{-1})^{-1}\prod_{p\;\text{prime}} (1-p^{-1}) = 0,$$
we may choose $b$ such that
$$\prod_{\substack{a\le p\le b\\ p\;\text{prime}}} (1-p^{-1}) \le \frac{D(S)}4.$$
Writing
$$S = (S\cap R(P_{a,b})) \cup \bigcup_{a\le p\le b} p S_p,$$
and applying (\ref{bon-1}) and Lemma~\ref{sieve}, we obtain
\begin{equation}
\label{point-bound}
\begin{split}
\sum_{\substack{a\le p\le b\\ p\;\text{prime}}} f_{S_p}(pt) 
= \sum_{\substack{a\le p\le b\\ p\;\text{prime}}} f_{pS_p}(t) &\ge f_S(t) - f_{S\cap R(P_{a,b})}(t) \\
&\ge f_S(t) - t^{-1}\prod_{\substack{a\le p\le b\\ p\;\text{prime}}}(1-p^{-1}).
\end{split}
\end{equation}
Let $X$ denote the set of values $T$ such that
$$\frac 1{\log T}\int_{1/T}^1 f_S(t)\,dt \ge \frac{3 D(S)}4.$$
By definition of $D(S)$, $X$ is unbounded.  Choosing $T\in X$ and integrating (\ref{point-bound}), we obtain
\begin{equation}
\label{int-bound}
\begin{split}
\frac 1{\log T}\int_{1/T}^1\sum_{\substack{a\le p\le b\\ p\;\text{prime}}} f_{S_p}(pt)\,dt 
&\ge \frac 1{\log T} \int_{1/T}^1 f_S(t)\,dt - \prod_{\substack{a\le p\le b\\ p\;\text{prime}}}(1-p^{-1}) \\
&\ge\frac{D(S)}2.
\end{split}
\end{equation}
Interchanging integral and sum and substituting $u=pt$, we obtain
$$\frac 1{\log T}\int_{1/T}^1\sum_{a\le p\le b} f_{S_p}(pt)\,dt
= \sum_{\substack{a\le p\le b\\ p\;\text{prime}}} \frac 1{p\log T} \int_{p/T}^p f_{S_p}(u)\,du.$$
By (\ref{upper}), this can expressed as
$$\sum_{\substack{a\le p\le b\\ p\;\text{prime}}} \frac 1{p\log T} \int_{1/T}^1 f_{S_p}(u)\,du + O\bigl(\frac 1{\log T}\bigr).$$
Combining this with (\ref{int-bound}) and choosing $T$ sufficiently large, we obtain
$$\frac 1{\log T}\sum_{\substack{a\le p\le b\\ p\;\text{prime}}} \int_{1/T}^1 f_{S_p}(u)\,du
\ge \frac 1{\log T}\sum_{\substack{a\le p\le b\\ p\;\text{prime}}} \frac ap \int_{1/T}^1 f_{S_p}(u)\,du
\ge \frac {a D(S)}3.$$
By definition of $a$, this means
\begin{equation}
\label{big-sum}
\frac 1{\log T}\sum_{\substack{a\le p\le b\\ p\;\text{prime}}} \int_{1/T}^1 f_{S_p}(u)\,du  \ge 4.
\end{equation}

Let
$$S' = \bigcup_{a\le p\le b} S_p.$$
We integrate the inequality (\ref{bon-2}) from $1/T$ to $1$ and interchange sum and integral
to obtain
\begin{equation}
\label{big-intersection}
\begin{split}
\sum_{\substack{a\le p<q\le b\\ p,q\;\text{prime}}}\frac 1{\log T} \int_{1/T}^1 f_{S_p\cap S_q}(u)\,du
\ge \frac 1{\log T}\sum_{\substack{a\le p\le b\\ p\;\text{prime}}} &\int_{1/T}^1 f_{S_p}(u)\,du \\
&- \frac 1{\log T} \int_{1/T}^1 f_{S'}(u)\,du.
\end{split}
\end{equation}
Applying (\ref{upper}) to $S'$ and integrating, we get
$$\frac 1{\log T} \int_{1/T}^1 f_{S'}(u)\,du \le 2.$$
Combining this with (\ref{big-sum}) and (\ref{big-intersection}), we get
$$\sum_{\substack{a\le p<q\le b\\ p,q\;\text{prime}}}\frac 1{\log T} \int_{1/T}^1 f_{S_p\cap S_q}(u)\,du \ge 2,$$
which implies
$$D(S_p\cap S_q) \ge \frac 1{(b-a)(1+b-a)}.$$
for some $p$ and $q$.
\end{proof}

The following theorem is a strengthened version of Theorem~\ref{paras}.

\begin{thm}
Let $\Sigma$ denote a finite set of primes
and $S$ a subset of $\bbn$ such that $D(S) > 0$. Then there exist a
rational number $c$ and an infinite set of pairwise distinct vectors $(a_{j,1},\ldots,a_{j,n})\in \bbq^n$
such that for each $j$,
$$\{c\prod_{i\in I} a_{j,i}\mid I\subset \{1,2,\ldots,n\}\} \subset S,$$
the numerators and denominators of the $a_{j,i}$ are not divisible by any prime in $\Sigma$, and the images of $a_{j,1},\ldots,a_{j,n}$ are linearly independent in $\bbq^*/{\bbq^*}^2$.
\end{thm}

\begin{proof}
We use induction on $n$.  The base case $n=1$ is trivial.  Suppose the theorem is true for some
$n\ge 1$.  Let $N$ be the maximum of all primes in $\Sigma$.
By Proposition~\ref{ind-step}, there exist primes $q>p>N$ such that $D(S_p\cap S_q) > 0$.
By the induction hypothesis for $S_p\cap S_q$ and $\Sigma' := \Sigma\cup \{p,q\}$, there exist
$c$ and $a_{j,i}$ prime to $\Sigma'$ such that the $a_{j,1},\ldots,a_{j,n}$ are independent in $\bbq^*/{\bbq^*}^2$
and for all $j$,
$$\{c\prod_{i\in I} a_{j,i}\mid I\subset \{1,2,\ldots,n\}\} \subset S_p\cap S_q.$$
Setting $c' := pc$, $a'_{j,1} := a_{j,1},\ldots,a'_{j,n}:=a_{j,n}$, and $a'_{j,n+1} := q/p$, we have
$$\{c'\prod_{i\in I} a'_{j,i}\mid I\subset \{1,2,\ldots,n+1\}\} \subset S,$$
and the $a'_{j,i}$ are prime to $\Sigma$.
\end{proof}

\section{The Main Theorem}

We prove Theorem~\ref{main}.

\noindent {\it Proof of Theorem~\ref{main}.} 
Let $E$ be the elliptic curve $X_0(19)$ and 
$$S:=\{d\in\bbn~~|~~ \text{rank}(E_d) > 0\}.$$
Then, Vatsal's theorem~\cite{Va} implies that $S$ has positive lower density. 

Let $n\in\bbn$ be given. By Theorem~\ref{paras}, there exist $c, a_1,\ldots, a_n\in\bbq^*$ such that the images of $a_i$ are linearly independent 
in $\bbq^*/{\bbq^*}^2$  and such that 
$$\{c\prod_{i\in I} a_i\mid I\subset \{1,2,\ldots,n\}\}\subset S.$$
If we let $V$ be a subspace of $\bbq^*/{\bbq^*}^2$ generated by $a_1,\ldots, a_n$, then it is of dimension $n$ and for each $v\in V$, the rank of the twist of $E_c$ by $v$ (which is $E_{cv}$) is positive.

\end{document}